\documentclass[a4paper,12pt]{article}
\usepackage{amsfonts}
\usepackage{amscd,color}
\usepackage{amsmath,amssymb,amscd}
\usepackage{indentfirst,graphicx,epsfig}
\input{epsf}
\usepackage{epstopdf}
\usepackage{caption}
\usepackage [latin1]{inputenc}
\newtheorem{thm}{Theorem}[section]

\newtheorem{lem}[thm]{Lemma}

\newtheorem{claim}[thm]{Claim}

\newenvironment {proof} {\noindent{\em Proof.}}{\hspace*{\fill}$\Box$\par\vspace{4mm}}
\newcommand{\ml}{l\kern-0.55mm\char39\kern-0.3mm}

\title{\textbf{Some results on the rainbow vertex-disconnection colorings of  graphs
}}
\author{{\small Yindi Weng } \\
{\small  Department of Mathematical Sciences, Zhejiang Sci-Tech University}\\
{\small  Hangzhou 310027, China}\\
{\small Email: wengyindi@zstu.edu.cn}\\
}
\date{}
\begin{document}
\maketitle
\begin{abstract}

Let $G$ be a nontrivial connected and vertex-colored graph. A vertex subset $X$ is called
rainbow if any two vertices in $X$ have distinct colors.
The graph $G$ is called \emph{rainbow vertex-disconnected} if for any two vertices $x$ and $y$
of $G$, there exists a vertex subset $S$ such that when $x$ and $y$ are nonadjacent,
$S$ is rainbow and $x$ and $y$ belong to different components of $G-S$; whereas when $x$ and $y$ are adjacent,
$S+x$ or $S+y$ is rainbow and $x$ and $y$ belong to different components of $(G-xy)-S$.
For a connected graph $G$, the \emph{rainbow vertex-disconnection number} of $G$, $rvd(G)$, is the minimum number of colors
that are needed to make $G$ rainbow vertex-disconnected.

In this paper,
we prove for any $K_4$-minor free graph,
$rvd(G)\leq \Delta(G)$ and the bound is sharp.
We show it is $NP$-complete to determine the rainbow vertex-disconnection number for bipartite graphs and split graphs.
Moreover,
we show for every $\epsilon>0$,
it is impossible to efficiently approximate the rainbow vertex-disconnection number of any bipartite graph and split graph within a factor of $n^{\frac{1}{3}-\epsilon}$ unless $ZPP=NP$.


\end{abstract}

\section{Introduction}

All graphs considered in this paper are simple, finite and undirected. Let $G=(V(G), E(G))$ be a
nontrivial connected graph with vertex set $V(G)$ and edge set $E(G)$. The $order$ of $G$ is denoted
by $n=|V(G)|$. For a vertex $v\in V$, the \emph{open neighborhood} of $v$ in $G$ is the set $N_{G}(v)=\{u\in V(G) | uv\in E(G)\}$ and
$d_{G}(v)=|N_{G}(v)|$ is the \emph{degree} of $v$ in $G$, and the \emph{closed neighborhood} in $G$ is the set $N_{G}[v]=N_{G}(v)\cup \{v\}$. The minimum and maximum
degree of $G$ are denoted by $\delta(G)$ and $\Delta(G)$, respectively. Let $P_n$ be a path with order $n$. We follow \cite{BM} for graph theoretical notation and terminology not defined here.

In graph theory, there are two ways (path and cut) to study the connectivity of graph.
For colored graphs, there are many concepts, such as rainbow connection coloring, proper connection coloring and so on, which study colored connectivity from colored paths.
Chartrand et al. \cite{GC} studied the rainbow edge-cut by introducing the concept
of rainbow disconnection of graphs.
They first researched the colored connectivity from the perspective of colored edge-cut.

Based on it,
Bai et al. \cite{BCLLW} researched the colored connectivity from the perspective of colored vertex-cut.
They introduced the concept of the rainbow vertex-disconnection
number, which can be applied to frequency assignment problem and the interception of goods.

For a connected and vertex-colored graph $G$, let $x$
and $y$ be two vertices of $G$.
If $x$ and $y$ are nonadjacent, then an $x$-$y$
\emph{vertex-cut} is a subset $S$ of
$V(G)$ such that $x$ and $y$ belong
to different components of $G-S$.
If $x$ and $y$ are adjacent, then an $x$-$y$
\emph{vertex-cut} is a subset $S$ of
$V(G)$ such that $x$ and $y$ belong
to different components of $(G-xy)-S$.
A vertex subset $S$ of
$G$ is \emph{rainbow} if no two vertices of $S$ have
the same color. An $x$-$y$ \emph{rainbow vertex-cut} is an $x$-$y$ vertex-cut $S$
such that if $x$ and $y$ are nonadjacent, then $S$ is rainbow; if $x$ and $y$ are adjacent,
then $S+x$ or $S+y$ is rainbow.

A vertex-colored graph $G$ is called \emph{rainbow
vertex-disconnected} if for any two vertices $x$ and $y$
of $G$, there exists an $x$-$y$ rainbow vertex-cut.
In this case, the vertex-coloring $c$ is
called a \emph{rainbow vertex-disconnection coloring}
of $G$. For a connected graph $G$, the \emph{rainbow vertex-disconnection number}
of $G$, denoted by $rvd(G)$, is the
minimum number of colors that are needed to make $G$
rainbow vertex-disconnected. A rainbow vertex-disconnection
coloring with $rvd(G)$ colors is called an
$rvd$-\emph{coloring} of $G$.

An \emph{injective coloring} of graph $G$ is a vertex-coloring of graph $G$ so that the colors of any two vertices with a common neighbor are different. The \emph{injective chromatic number} $\chi_i(G)$ of a graph $G$ is the minimum number of colors such that there is an injective coloring. The injective coloring was first introduced by Hahn et al. in 2002 \cite{HKS} and originated from the complexity theory on Random Access Machines.

According to \cite{LW}, we have $$\delta(G)\leq rvd(G)\leq \chi_i(G)\leq \Delta(G)(\Delta(G)-1)+1.$$

A \emph{minor} of a graph $G$ is any graph obtainable from $G$ by means of a sequence of vertex and edge deletions and edge contractions.
We call $G$ \emph{$H$-minor free} if $G$ does not have $H$ as a minor.
A \emph{split} graph is a graph whose vertices
can be partitioned into a clique and an independent set.

Chen et al. \cite{CHRW} proved that every $K_4$-minor free graph $G$ with maximum degree $\Delta\geq 1$ has $\chi_i(G)\leq \lceil\frac{3}{2}\Delta\rceil$.
Jin et al. \cite{JXZ} considered the complexity of injective coloring.
They also showed if $ZPP\neq NP$, then for every $\epsilon>0$,
it is not possible to efficiently approximate $\chi_i(G)$ within a factor of $n^{\frac{1}{3}-\epsilon}$ for any bipartite graph $G$.
For rainbow vertex-disconnection colorings of graphs,
Chen et al. \cite{CLLW} showed that it is NP-complete to
decide whether a given vertex-colored graph G is rainbow vertex-disconnected, even though the graph $G$ has $\Delta(G)=3$ or is bipartite.
But how about the complexity of determining $rvd(G)$?

Inspired by these,
the paper is organized as follows.
In Section 2,
we consider the rainbow vertex-disconnection numbers of $K_4$-minor free graphs.
We prove for any $K_4$-minor free graph,
$rvd(G)\leq \Delta(G)$ and the bound is sharp.
It shows that there is a certain gap between $rvd(G)$ and $\chi_i(G)$ even if $G$ is $K_4$-minor free.
In Section 3,
we prove it is $NP$-complete to determine the rainbow vertex-disconnection number for bipartite graphs and split graphs.
Moreover,
we show for every $\epsilon>0$,
it is impossible to efficiently approximate the rainbow vertex-disconnection number of any bipartite graph and split graph within a factor of $n^{\frac{1}{3}-\epsilon}$ unless $ZPP=NP$.

\section{Graphs with $K_4$-minor free}

\begin{lem}\cite{BCLLW}\label{rvd1}
Let $G$ be a nontrivial connected graph. Then $rvd(G)=1$
if and only if $G$ is a tree.
\end{lem}

\begin{lem}\cite{BCLLW}\label{rvdcycle}
If $C_{n}$ is a cycle of order $n\geq 3$, then
$rvd(C_{n})=2$.
\end{lem}

\begin{lem}\cite{BCLLW}\label{rvdblock}
Let $G$ be a nontrivial connected graph, and let $B$ be
a block of $G$ such that $rvd(B)$ is maximum among all
blocks of $G$. Then $rvd(G)=rvd(B)$.
\end{lem}

\begin{lem}\label{rvddifcolor}\cite{BCLLW}
Let $G$ be a nontrivial connected graph, and let $u$ and $v$ be two vertices of $G$ having at least two common neighbors.
Then $u$ and $v$ receive different colors in any rvd-coloring of $G$.
\end{lem}

In fact, Lemma \ref{rvddifcolor} holds true for any rainbow vertex-disconnection coloring of $G$.
For convenience, we think Lemma \ref{rvddifcolor} is for any rainbow vertex-disconnection coloring of $G$.

Let $S_G(x,y)$ be an $x$-$y$ rainbow vertex-cut in $G$.
Let $D_G(x,y)$ be the rainbow vertex set such that
if $x,y$ are adjacent,
then $S_G(x,y)+x\subseteq D_G(x,y)$ or $S_G(x,y)+y\subseteq D_G(x,y)$
and $D_G(x,y)$ is rainbow;
if $x,y$ are not adjacent,
then $S_G(x,y)\subseteq D_G(x,y)$
and $D_G(x,y)$ is rainbow.
In order to prove that a vertex-coloring of $G$ is a rainbow vertex-disconnection coloring,
for any two vertices $x,y$ of $G$,
we only need to find $D_G(x,y)$.
Every $K_4$-minor free graph contains a vertex with degree at most two\cite{D}.

We call $v$ a \emph{k-vertex} if $d_G(v)=k$. Define $T_G(u)=\{x|d_G(x)\geq 3$ such that either $ux\in E(G)$,
or there exists a $2$-vertex $z$ satisfying $uz,zx\in E(G)\}$.
Let $t_G(u)=|T_G(u)|$.
Let $x$ and $y$ be two vertices of graph $G$.
The set of all the 2-vertices which are adjacent to both $x$ and $y$ is denoted by $M_G(x,y)$.
Let $m_G(x,y)=|M_G(x,y)|$.
Let $u\sim v$ ($u\not\sim v$) denote that vertex $u$ and vertex $v$ are adjacent (not adjacent) in $G$.

Lih et al. \cite{LWZ} proved the following Lemma.

\begin{lem}\cite{LWZ}\label{rvdk4mlem}
Let $G$ be a $K_4$-minor free graph. Then one of the following holds:

$(i)$ $\delta(G)\leq 1$;

$(ii)$ there exist two adjacent $2$-vertices;

$(iii)$ there exists a vertex $u$ with $d_G(u)\geq 3$ such that $t_G(u)\leq 2$.
\end{lem}

To \emph{contract} an
edge $e$ of a graph $G$ is to delete the edge and then identify
its ends.
The resulting graph is denoted by $G/e$.

\begin{lem}\label{adjacenttwo}
Let $G$ be a 2-connected graph with order $n\geq 4$ and two adjacent 2-vertices $u,v$.
Then $rvd(G)\leq rvd(G/uv)$, where $uv$ is an edge of $G$.

\end{lem}
\begin{proof}
Let $N_G(u)=\{u_1,v\}$ and $N_G(v)=\{v_1,u\}$.
Since $G$ is 2-connected,
we consider $u_1\neq v_1$.
For convenience,
regard $G/uv$ as the graph $H$ obtained from graph $G$ by deleting the vertex $v$ and adding the edge $uv_1$.
Since $H$ is also 2-connected,
we have $rvd(H)\geq 2$.
Let $c_H$ be an rvd-coloring of $H$ and $|c_H|$ be the number of colors.

Consider there exist at least two colors in $\{u,u_1,v_1\}$ under $c_H$.
We extend $c_H$ to a vertex-coloring $c_G$ of graph $G$ as follows.
If $c_H(u)\neq c_H(v_1)$, color $v$ different from $c_H(u_1)$ and color $V(G)\setminus\{v\}$ with the same colors from $c_H$.
If $c_H(u)=c_H(v_1)$,
let $c_G(u)=c_H(u_1)$, $c_G(v)=c_H(v_1)$
and color $V(G)\setminus\{u,v\}$ with the same colors from $c_H$.
Obviously,
$N_G(u)$ and $N_G(v)$ are rainbow.
Now we claim $c_G$ is a rainbow vertex-disconnection coloring of $G$.
Let $x$ and $y$ be two vertices of graph $G$.
If $x\in\{u,v\}$, then $D_G(x,y)=N_G(x)$.
By symmetry,
consider $x,y\in V(G)\setminus\{u,v\}$.
If $u\in D_H(x,y)$,
then $D_G(x,y)=D_H(x,y)$ or $D_H(x,y)\cup\{v\}\setminus\{u\}$.
If $u\not\in D_H(x,y)$,
then $D_G(x,y)=D_H(x,y)$.
So $rvd(G)\leq |c_G|=|c_H|= rvd(H)$.

Consider $c_H(u)=c_H(u_1)=c_H(v_1)$ under $c_H$.
Then $u_1\not\sim v_1$ in $H$ and $G$.
We extend $c_H$ to a vertex-coloring $c_G$ of graph $G$ as follows.
Color $v$ different from $c_H(u)$ and color $V(G)\setminus\{v\}$ with the same colors from $c_H$.
Let $x$ and $y$ be two vertices of graph $G$.
If $x=u$, then $D_G(x,y)=N_G(x)$.
By symmetry,
consider $x,y\in V(G)\setminus\{u\}$.
We have $D_G(v,u_1)=D_H(u_1,v_1)$.
If $x=v$ and $y\in V(G)\setminus\{u,v,u_1\}$,
then $D_G(x,y)=D_H(u,y)$.
By symmetry,
consider $x,y\in V(G)\setminus\{u,v\}$.
We have $D_G(x,y)=D_H(x,y)$.
So $c_G$ is a rainbow vertex-disconnection coloring of $G$ and $rvd(G)\leq rvd(H)$.
\end{proof}

\begin{thm}
Let $G$ be a $K_4$-minor free graph. Then $rvd(G)\leq \Delta(G)$ and the bound is sharp.
\end{thm}
\begin{proof}
When $\Delta(G)\leq 2$, the graph $G$ is a path or cycle. By Lemmas \ref{rvd1} and \ref{rvdcycle}, we have $rvd(G)\leq \Delta(G)$.
So consider $\Delta(G)\geq 3$.

We prove the result by induction on the order of graph $G$. When $n=4$, the graph $G$ is $K_4-{e}$ or a triangle with one pendant edge or $K_{1,3}$.
Obviously, $rvd(G)\leq 3$.
Assume $n\geq 5$ and the theorem holds for any $K_4$-minor free graph $\widetilde{G}$ with $|\widetilde{G}|<|G|$.
By Lemma \ref{rvdblock},
we only need to consider that $G$ is 2-connected and $\delta(G)\geq 2$.
If $G$ has two adjacent $2$-vertices $u$ and $v$,
then $rvd(G)\leq rvd(G/uv)\leq \Delta(G/uv)\leq \Delta(G)$ by Lemma \ref{adjacenttwo} and induction hypothesis.
So by Lemma \ref{rvdk4mlem},
consider that $G$ has no adjacent $2$-vertices and there exists a vertex $u$ with $d_G(u)\geq 3$ such that $t_G(u)\leq 2$.

If $t_G(u)=0$,
all the neighbors of $u$ are 2-vertices and there exist adjacent $2$-vertices.
It is a contradiction.
If $t_G(u)=1$, assuming that $T_G(u)=\{u_1\}$, then all the neighbors of $u$ are $u_1$ or some neighbors of $u_1$.
Since $G$ is 2-connected,
$G$ is $K_{2,n-2}$ or $K_{2,n-2}+\{uu_1\}$.
Obviously,
$rvd(G)=\Delta(G)$.

So $t_G(u)=2$ and any vertex $v$ with degree at least $3$ has $t_G(v)\geq 2$.
Assume that $T_G(u)=\{u_1,u_2\}$. Then all the neighbors of $u$ are $u_1$, $u_2$ or some neighbors of $u_1$ or $u_2$.

Let $S=\{v|d_G(v)\geq 3$ and $t_G(v)=2\}$.
For any vertex $v\in S$,
let $T_G(v)=\{v_1,v_2\}$.

\begin{claim}\label{minorclaim}
There exists a vertex $v$ in $S$ satisfying that if $v_1\sim v_2$,
then $t_G(v_1)\in\{2,3\}$ or $t_G(v_2)\in\{2,3\}$;
if $v_1\not\sim v_2$,
then $t_G(v_1)=2$ or $t_G(v_2)=2$.
\end{claim}
\begin{proof}
Suppose not.
For any vertex $v\in S$,
if $v_1\sim v_2$,
then $t_G(v_1)\geq 4$ and $t_G(v_2)\geq 4$;
if $v_1\not\sim v_2$,
then $t_G(v_1)\geq 3$ and $t_G(v_2)\geq 3$.
We construct a new graph $H$ from $G$.
Let $V(H)=\{v|d_G(v)\geq 3\}$
and $E(H)=\{xy|x,y\in V(H)$ and $y\in T_G(x)\}$.
Obviously,
for any 2-vertex $v$ in $H$,
if $v_1\sim v_2$, then $d_H(v_1)\geq 4$ and $d_H(v_2)\geq 4$;
if $v_1\not\sim v_2$,
then $d_H(v_1)\geq 3$ and $d_H(v_2)\geq 3$.
For every 2-vertex $v$ in $H$,
contract the edge $vv_1$ in $H$.
We obtain a new graph $H'$ from $H$ with minimum degree at least three.
Then $H'$ is not $K_4$-minor free.
It is a contradiction.
\end{proof}

Reselect vertex $u$ with $t_G(u)=2$ satisfying Claim \ref{minorclaim}.
Without loss of generality,
assume that $t_G(u_1)\in\{2,3\}$ for $u_1\sim u_2$ and $t_G(u_1)=2$ for $u_1\not\sim u_2$.
Consider $u_1\sim u_2$ and $t_G(u_1)=2$.
If $t_G(u_2)\geq 3$,
then $u_2$ is a cut vertex, which is a contradiction.
If $t_G(u_2)=2$,
then the graph $G$ is as shown in figure \ref{k4minor}.
We give a vertex-coloring $c_G$ of $G$ as follows.
Let $c_G(u)=1,c_G(u_1)=2$ and $c_G(u_2)=3$.
Color $M_G(u_1,u_2)$ different from $u,u_2$ and rainbow.
If $u\not\sim u_1$,
color $M_G(u,u_1)$ different from $u_2$ and rainbow;
otherwise,
color $M_G(u,u_1)$ different from $u,u_2$ and rainbow.
If $u\not\sim u_2$,
color $M_G(u,u_2)$ different from $u_1$ and rainbow;
otherwise,
color $M_G(u,u_2)$ different from $u,u_1$ and rainbow.
Obviously,
$c_G$ is a rainbow vertex-disconnection coloring of $G$ with at most $\Delta(G)$ colors.
So $rvd(G)\leq \Delta(G)$.
Thus, $t_G(u_1)=3$ for $u_1\sim u_2$ and $t_G(u_1)=2$ for $u_1\not\sim u_2$.

\begin{figure}[ht]
\centering
\includegraphics[scale=0.9]{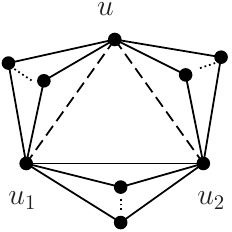}
\caption{The graph with $t_G(u_2)=2$.}\label{k4minor}
\end{figure}

Let $s_1\in T_G(u_1)$ and $s_1\neq u,u_2$.
Let $Q_u$ be the set of neighbors of $u$ with degree two.
Since $d_G(u)\geq 3$,
there exists at least one neighbor of $u$ with degree two.
Let $H$ be the graph obtained from $G$ by deleting $Q_u$ and adding edges to ensure $u\sim u_1$ and $u\sim u_2$.

\begin{claim}\label{reop}
There exists a rainbow vertex-disconnection coloring $c_H$ of $H$ with at most $\Delta(H)$ colors
such that at least two vertices from $\{u,u_1,u_2\}$ have different colors.
\end{claim}
\begin{proof}
Assume, to the contrary,
we have $c(u)=c(u_1)=c(u_2)$ for any rainbow vertex-disconnection coloring $c$ of $H$ with at most $\Delta(H)$ colors.
Then $u_1\not\sim u_2$ and $m_H(u_1,u_2)=1$ by Lemma \ref{rvddifcolor}.
Let $c_H$ be a rainbow vertex-disconnection coloring of $H$ with at most $\Delta(H)$ colors.
We construct a new coloring $c_{H'}$ of $H$ as follows.
Let $c_{H'}(v)=c_H(v)$ for $v\in V(H)\setminus\{u\}$.
Color $u$ such that $c_{H'}(u)\neq c_H(u_1),c_H(s_1)$.
We denote the new colored graph by $H'$.
Now claim $c_{H'}$ is a rainbow vertex-disconnection coloring of $H'$.

Let $x$ and $y$ be two vertices of graph $H'$.
We have $D_{H'}(u,u_1)=\{u,u_2\}$,$D_{H'}(u,u_2)=\{u,u_1\}$
and $D_{H'}(u_1,u_2)=\{u,s_1\}$.
When $x=u_1$ and $y\in V(H')\setminus\{u,u_1,u_2\}$,
if $u\in D_H(x,y)$,
then $D_{H'}(x,y)=D_H(x,y)\cup\{u_2\}\setminus\{u\}$;
otherwise,
$D_{H'}(x,y)=D_{H}(x,y)$.
By symmetry,
consider $x,y\in V(H')\setminus\{u_1\}$ and $\{x,y\}\neq \{u,u_2\}$.
If $u\in D_H(x,y)$,
then $D_{H'}(x,y)=D_H(x,y)\cup\{u_1\}\setminus\{u\}$;
otherwise,
$D_{H'}(x,y)=D_{H}(x,y)$.
So $c_{H'}$ is a rainbow vertex-disconnection coloring of $H'$ with at most $\Delta(H')$ colors
and $c_{H'}(u)\neq c_{H'}(u_1)$.
It is a contradiction.

\end{proof}

\begin{claim}\label{rvdextend}
Let $u_1\not\sim u_2$ and $c_H$ be a rainbow vertex-disconnection coloring of $H$ from Claim \ref{reop}.
If there exists a vertex $u_i$ from $\{u_1,u_2\}$ satisfying  $m_G(u,u_i)\leq 1$ and $c_H(u)=c_H(u_i)$,
then $c_H$ can be extended to a rainbow vertex-disconnection coloring $c_G$ of $G$ with at most $\Delta(G)$ colors.
\end{claim}
\begin{proof}
Assume that $c_H(u_1)=c_H(u)\neq c_H(u_2)$.
If $m_G(u,u_1)=1$,
then let $M_G(u,u_1)=\{t\}$.
We extend $c_H$ to a coloring $c_G$ of $G$ as follows.
Let $c_G(v)=c_H(v)$ for $v\in V(H)$.
Color $t$ different from $u_1,u_2$.
Color $M_G(u,u_2)$ different from $N_G(u)\setminus M_G(u,u_2)$ and rainbow.
Then $c_G$ has at most $\Delta(G)$ colors.
Now we claim that $c_{G}$ is a rainbow vertex-disconnection coloring of $G$.
Let $x$ and $y$ be two vertices of graph $G$.
If $x\in \{u\}\cup M_G(u,u_2)$,
then $D_G(x,y)=N_G(x)$.
By symmetry,
consider $x,y\not\in \{u\}\cup M_G(u,u_2)$.
Assume that $x=t$.
We have $D_G(t,u_1)=\{u,t\}$ and $D_G(t,u_2)=D_{H}(u_1,u_2)$.
If $y\not\in \{u,t,u_1,u_2\}\cup M_G(u,u_2)$,
$D_G(t,y)=\{u_1,u_2\}$.
By symmetry, assume that $x,y\not\in \{u,t\}\cup M_G(u,u_2)$.
We have $D_G(x,y)=D_H(x,y)$.
So $c_G$ is a rainbow vertex-disconnection coloring of $G$ with at most $\Delta(G)$ colors.

If $m_G(u,u_1)=0$,
it is similar to the above coloring $c_G$ without $t$.
The case $c_H(u_2)=c_H(u)\neq c_H(u_1)$ will not be repeated here.
\end{proof}

By Claim \ref{reop}, there are four cases under $c_H$ in $H$.

\textbf{Case $1.$ } $\{u,u_1,u_2\}$ is rainbow.

We extend $c_H$ to a vertex-coloring $c_G$ of $G$ as follows.
Let $c_G(v)=c_H(v)$ for $v\in V(H)$.
Color $Q_u$ different from $N_G(u)\setminus Q_u$ and rainbow.
Now we claim that $c_{G}$ is a rainbow vertex-disconnection coloring of $G$.
Let $x$ and $y$ be two vertices of graph $G$.
If $x\in  \{u\}\cup Q_u$,
then $D_G(x,y)=N_G(x)$.
By symmetry, consider $x,y\in V(G)\setminus \{\{u\}\cup Q_u\}$.
Then $D_G(x,y)=D_{H}(x,y)$.
So $c_G$ is a rainbow vertex-disconnection coloring of $G$ with at most $\Delta(G)$ colors.

\textbf{Case $2.$} $c_H(u_1)=c_H(u_2)\neq c_H(u)$.

We extend $c_H$ to a vertex-coloring $c_G$ of $G$ as follows.
Let $c_G(v)=c_H(v)$ for $v\in V(H)$.
If $u\sim u_1$,
color $M_G(u,u_1)$ different from $u,u_2$ and rainbow;
otherwise,
color $M_G(u,u_1)$ different from $u_2$ and rainbow.
If $u\sim u_2$,
color $M_G(u,u_2)$ different from $u,u_1$ and rainbow;
otherwise,
color $M_G(u,u_2)$ different from $u_1$ and rainbow.
Then $c_G$ uses at most $\Delta(G)$ colors.

Now we claim that $c_{G}$ is a rainbow vertex-disconnection coloring of $G$.
Let $x$ and $y$ be two vertices of graph $G$.
If $x\in Q_u$,
then $D_G(x,y)=N_G(x)$.
By symmetry, consider $x,y\in V(G)\setminus Q_u$.
Consider $x=u$.
We have $D_G(u,u_1)=M_G(u,u_1)\cup\{u_2\}$ or $M_G(u,u_1)\cup\{u,u_2\}$
and $D_G(u,u_2)=M_G(u,u_2)\cup\{u_1\}$ or $M_G(u,u_2)\cup\{u,u_1\}$.
If $y\not\in  Q_u\cup\{u,u_1,u_2\}$,
then $D_G(u,y)=D_{H}(u,y)$.
By symmetry, consider $x,y\in V(G)\setminus \{\{u\}\cup Q_u\}$.
Then $D_G(x,y)=D_{H}(x,y)$.
So $c_G$ is a rainbow vertex-disconnection coloring of $G$ with at most $\Delta(G)$ colors.

\begin{claim}\label{rvdmg}
 Assume that $c_H(u)=c_H(u_1)\neq c_H(u_2)$ or $c_H(u)=c_H(u_2)\neq c_H(u_1)$.
 If $m_H(u_1,u_2)\geq 2$,
 then there exists a rainbow vertex-disconnection coloring of $G$ with at most $\Delta(G)$ colors.
\end{claim}
\begin{proof}
Assume that $c_H(u)=c_H(u_1)\neq c_H(u_2)$.
If there exists a vertex $u_0\in M_H(u_1,u_2)$ with the color different from $c_H(u_1)$ and $c_H(u_2)$,
by symmetry,
we regard $u_0$ as $u$.
It belongs to Case 1.
So $m_H(u_1,u_2)=2$ and $M_H(u_1,u_2)$ has the same colors with $u_1$ and $u_2$.
Then $u_1\not\sim u_2$.
By Claim \ref{rvdextend},
$m_G(u,u_1)\geq 2$.
So $\Delta(G)\geq 4$,
otherwise $u_2$ is a cut vertex,
which is a contradiction.
Let $M_H(u_1,u_2)=\{u,u'\}$.
We give a new vertex-coloring $c_{H'}$ of $H$ by recoloring $u$.
If $c_H(s_1)=c_H(u_1)$ or $c_H(u_2)$,
we regard the vertex from $M_H(u_1,u_2)$ with the same color of $s_1$ as vertex $u$ and color $u$ different from $u_1,u_2$.
Otherwise,
color $u$ different from $u_1,u_2,s_1$.
We denote the new colored graph by $H'$.
Then $c_{H'}$ uses at most $\Delta(H')$ colors.
Now we claim that $c_{H'}$ is a rainbow vertex-disconnection coloring of $H'$.
Let $x$ and $y$ be two vertices of graph $H'$.
If $x\in \{u,u'\}$, then $D_{H'}(x,y)=N_{H'}(x)$.
We have $D_{H'}(u_1,u_2)=\{u,u',s_1\}$.
For $x=u_1$ and $y\not\in\{u,u',u_1,u_2\}$,
if $\{u,u'\}\subseteq D_H(x,y)$,
then $D_{H'}(x,y)=D_H(x,y)\cup\{u_2\}\setminus\{u,u'\}$;
otherwise,
$D_{H'}(x,y)=D_H(x,y)\setminus\{u,u'\}$.
When $\{x,y\}$ is other pairs of vertices,
if $\{u,u'\}\subseteq D_H(x,y)$,
then $D_{H'}(x,y)=D_H(x,y)\cup\{u_1\}\setminus\{u,u'\}$;
otherwise,
$D_{H'}(x,y)=D_H(x,y)\setminus\{u,u'\}$.
So $c_{H'}$ is a rainbow vertex-disconnection coloring of $H'$ with at most $\Delta(H')$ colors, where $\{u,u_1,u_2\}$ is rainbow.
It belongs to Case 1.
The case $c_H(u)=c_H(u_2)\neq c_H(u_1)$ will not be repeated here.
\end{proof}

\textbf{Case $3.$} $c_H(u)=c_H(u_1)\neq c_H(u_2)$.

By Claim \ref{rvdmg},
we have $m_H(u_1,u_2)=1$.
Consider $\Delta(G)\geq 4$.
Assume that there exists $D_H(u_1,s_1)$ such that $u\not\in D_H(u_1,s_1)$.
We give a new vertex-coloring $c_{H'}$ of $H$ by recoloring $u$ different from $s_1,u_1,u_2$. We denote the new colored graph by $H'$.
Then $c_{H'}$ uses at most $\Delta(H')$ colors.
Now we claim that $c_{H'}$ is a rainbow vertex-disconnection coloring of $H'$.
Let $x$ and $y$ be two vertices of graph $H'$.
If $x=u$, then $D_{H'}(x,y)=N_{H'}(u)$.
We have $D_{H'}(u_1,s_1)=D_H(u_1,s_1)$ and $D_{H'}(u_1,u_2)=\{u,s_1,u_1\}$ or $\{u,s_1,u_2\}$.
If $x=u_1$ and $y$ is the neighbor of $u_1$ with degree two,
then $D_{H'}(u_1,y)=\{s_1,u_1\}$ or $\{s_1,y\}$ or $\{u_2,u_1\}$.
If $x=u_1$ and $y\not\in N_{H'}[u_1]\cup\{s_1,u_2\}$,
then $D_{H'}(u_1,y)=\{u,s_1\}$ for $u_1\not\sim u_2$ and $D_{H'}(u_1,y)=D_{H}(u_1,y)\setminus\{u\}$ for $u_1\sim u_2$.
By symmetry,
consider $x,y\not\in \{u,u_1\}$.
If $u\in D_H(x,y)$,
then $D_{H'}(x,y)=D_H(x,y)\cup\{u_1\}\setminus\{u\}$;
otherwise,
$D_{H'}(x,y)=D_H(x,y)$.
So $c_{H'}$ is a rainbow vertex-disconnection coloring of $H'$ with at most $\Delta(H')$ colors, where $\{u,u_1,u_2\}$ is rainbow.
It belongs to Case 1.

Assume that $u$ is contained in any $u_1$-$s_1$ rainbow vertex-cut under $c_H$ in $H$.
Then $M_{H}(u_1,s_1)$ has no color like $u$ under $c_H$ and $u_1\not\sim u_2$.
$N_H(u_1)$ is rainbow.
The color of $u_2$ appears in the colors of $M_H(u_1,s_1)$.
Otherwise,
let $D_H(u_1,s_1)=M_H(u_1,s_1)\cup\{u_1,u_2\}$ containing no $u$, which is a contradiction.
If $u_1\sim s_1$,
then the color of $s_1$ is different from the colors of $u_1$ and $u_2$ under $c_H$.
Based on $c_H$,
we recolor $u$ different from the colors of $N_H(u_1)$.
We denote the new colored graph by $H'$ and the vertex-coloring by $c_{H'}$.
Then $\{u,u_1,u_2\}$ is rainbow in $H'$.
By Claim \ref{rvdextend},
we consider $m_G(u,u_1)\geq 2$.
We have $\Delta(G)\geq d_G(u_1)\geq d_H(u_1)+1$.
So we use at most $\Delta(G)$ colors under $c_{H'}$.
Now we claim that $c_{H'}$ is a rainbow vertex-disconnection coloring of $H'$.
Let $x$ and $y$ be two vertices of graph $H'$.
If $x=u$ or $u_1$, then $D_{H'}(x,y)=N_{H'}(x)$.
By symmetry, consider $x,y\in V(H')\setminus\{u,u_1\}$.
If $u\in D_H(x,y)$,
then $D_{H'}(x,y)=D_H(x,y)\cup\{u_1\}\setminus\{u\}$;
otherwise,
$D_{H'}(x,y)=D_H(x,y)$.
So $c_{H'}$ is a rainbow vertex-disconnection coloring of $H'$.
It belongs to Case 1.

Consider $\Delta(G)=3$.
Then $d_G(u)=3$.
If $u_1\sim u_2$,
we have $N_G[u]\cup\{u_1,u_2\}$ is a block and $G$ is not 2-connected.
It is a contradiction.

Assume that $u_1\not\sim u_2$.
By Claim \ref{rvdextend},
consider $m_G(u,u_1)=2$.
We have $d_H(u_1)=2$.
Based on $c_H$,
we recolor $u$ different from the colors of $N_H(u_1)$.
We denote the new colored graph by $H'$ and the vertex-coloring by $c_{H'}$.
Then we use at most three colors in $c_{H'}$ and $N_{H'}(u_1)$ is rainbow.
Now we claim that $c_{H'}$ is a rainbow vertex-disconnection coloring of $H'$.
Let $x$ and $y$ be two vertices of graph $H'$.
If $x=u$ or $u_1$, then $D_{H'}(x,y)=N_{H'}(x)$.
By symmetry, consider $x,y\in V(H')\setminus\{u,u_1\}$.
If $u\in D_H(x,y)$,
then $D_{H'}(x,y)=D_H(x,y)\cup\{u_1\}\setminus\{u\}$;
otherwise,
$D_{H'}(x,y)=D_H(x,y)$.
So $c_{H'}$ is a rainbow vertex-disconnection coloring of $H'$.
It belongs to Case 1 or has $c_{H'}(u)=c_{H'}(u_2)$.
If $c_{H'}(u)=c_{H'}(u_2)$,
since $m_G(u,u_2)\leq 1$,
there exists a rainbow vertex-disconnection coloring $c_G$ of $G$ with at most $\Delta(G)$ colors by Claim \ref{rvdextend}.

\textbf{Case $4.$} $c_H(u)=c_H(u_2)\neq c_H(u_1)$.

By Claim \ref{rvdmg},
we have $m_H(u_1,u_2)=1$.
Assume that $u_1\sim u_2$.
Based on $c_H$,
we give a new vertex-coloring $c_{H'}$ of $H$ by recoloring $u$.
If $c_H(s_1)\neq c_H(u_2)$,
recolor $u$ different from $s_1,u_2$;
otherwise,
recolor $u$ different from $s_1,u_1$.
We denote the new colored graph by $H'$.
Then $c_{H'}(u)\neq c_{H'}(u_2)$.

Now we claim that $c_{H'}$ is a rainbow vertex-disconnection coloring of $H'$.
Let $x$ and $y$ be two vertices of graph $H'$.
If $x=u$, then $D_{H'}(x,y)=N_{H'}(u)$.
By symmetry, consider $x,y\in V(H')\setminus\{u\}$.
We have $D_{H'}(u_1,u_2)=\{u,s_1,u_1\}$ or $\{u,s_1,u_2\}$.
When $\{x,y\}$ is other pairs of vertices,
if $u\in D_H(x,y)$,
then $D_{H'}(x,y)=D_H(x,y)\setminus\{u\}$;
otherwise,
$D_{H'}(x,y)=D_H(x,y)$.
So $c_{H'}$ is a rainbow vertex-disconnection coloring of $H'$.
It belongs to Case 1 or 3.

Assume that $u_1\not\sim u_2$.
By Claim \ref{rvdextend},
we consider $m_G(u,u_2)\geq 2$.
If $m_G(u,u_1)\geq 1$,
we can restrict an rvd-coloring of $G-M_G(u,u_1)+\{uu_1\}$ to $H$.
Then it belongs to Case 1 or 2 or 3.
So $m_G(u,u_1)=0$.
Then $u\sim u_1$ in $G$.
We have $m_G(u_1,s_1)\geq 2$.
Otherwise,
regarding $u_1$ as $u$,
by Claim \ref{rvdextend} and Case 1 and Case 2,
there exists a rainbow vertex-disconnection coloring $c_G$ of $G$ with at most $\Delta(G)$ colors.
Let $G'$ be the graph obtained from $G$ by deleting $M_G(u,u_2)$, $M_G(u_1,s_1)$, $\{u,u_1\}$ and adding two new vertices $q_1,q_2$,
which are the common neighbors of $s_1$ and $u_2$.
By induction hypothesis,
there exists a rainbow vertex-disconnection coloring $c_{G'}$ of $G'$ using at most $\Delta(G')$ colors.
Obviously, $c_{G'}(s_1)\neq c_{G'}(u_2)$.

Assume that $s_1\sim u_2$.
Then $\{s_1,u_2,q_1\}$ or $\{s_1,u_2,q_2\}$ is rainbow under $c_{G'}$.
Without loss of generality,
assume that $\{s_1,u_2,q_1\}$ is rainbow.
Now we extend $c_{G'}$ to a coloring $c_G$ of $G$ as follows.
Let $c_G(v)=c_{G'}(v)$ for $v\in V(G')\setminus\{q_1,q_2\}$.
Let $c_G(u_1)=c_{G'}(q_1)$.
Color $u$ different from $s_1$ and $u_2$.
Color $M_G(u,u_2)$ different from $N_G(u)\setminus M_G(u,u_2)$ and rainbow.
Color $M_G(u_1,s_1)$ different from $N_G(u_1)\setminus M_G(u_1,s_1)$ and rainbow.
Now we claim that $c_{G}$ is a rainbow vertex-disconnection coloring of $G$.
Let $x$ and $y$ be two vertices of graph $G$.
If $x\in \{u,u_1\}\cup M_G(u,u_2)\cup M_G(u_1,s_1)$, then $D_{G}(x,y)=N_{G}(x)$.
By symmetry, consider $x,y\not\in \{u,u_1\}\cup M_G(u,u_2)\cup M_G(u_1,s_1)$.
We have $D_G(s_1,u_2)=D_{G'}(s_1,u_2)\cup\{u_1\}\setminus\{q_1,q_2\}$.
When $\{x,y\}$ is other pairs of vertices,
$D_G(x,y)=D_{G'}(x,y)\setminus\{q_1,q_2\}$.
So $c_{G}$ is a rainbow vertex-disconnection coloring of $G$ using at most $\Delta(G)$ colors.

Assume that $s_1\not\sim u_2$.
Then $c_{G'}(q_1)\neq c_{G'}(u_2)$ or $c_{G'}(q_2)\neq c_{G'}(u_2)$.
Without loss of generality,
assume that $c_{G'}(q_1)\neq c_{G'}(u_2)$.
Now we extend $c_{G'}$ to a vertex-coloring $c_G$ of $G$ as follows.
Let $c_G(v)=c_{G'}(v)$ for $v\in V(G')\setminus\{q_1,q_2\}$.
Let $c_G(u)=c_{G'}(q_1)$.
Color $u_1$ different from $s_1$ and $u_2$.
Color $M_G(u,u_2)$ different from $N_G(u)\setminus M_G(u,u_2)$ and rainbow.
If $u_1\sim s_1$,
color $M_G(u_1,s_1)$ different from $\{s_1,u_2\}$ and rainbow;
otherwise,
color $M_G(u_1,s_1)$ different from $u_2$ and rainbow.
Then $c_{G}$ uses at most $\Delta(G)$ colors.
Now we claim that $c_{G}$ is a rainbow vertex-disconnection coloring of $G$.
Let $x$ and $y$ be two vertices of graph $G$.
Let $T=\{u\}\cup M_G(u,u_2)\cup M_G(u_1,s_1)$.
If $x\in T$, then $D_{G}(x,y)=N_{G}(x)$.
By symmetry, consider $x,y\not\in T$.
We have $D_G(u_1,u_2)=D_{G'}(s_1,u_2)\cup\{u\}\setminus\{q_1,q_2\}$
and $D_G(u_1,s_1)=N_G(u_1)\cup\{u_2\}\setminus\{u\}$.
If $x=u_1$ and $y\not\in T\cup\{u_1,u_2,s_1\}$,
we have $D_G(u_1,y)=\{s_1,u_2\}$.
By symmetry, consider $x,y\not\in T\cup\{u_1\}$.
If $q_1\in D_{G'}(x,y)$,
then $D_G(x,y)=D_{G'}(x,y)\cup\{u\}\setminus\{q_1,q_2\}$;
otherwise,
$D_G(x,y)=D_{G'}(x,y)\setminus\{q_2\}$.
So $c_{G}$ is a rainbow vertex-disconnection coloring of $G$ using at most $\Delta(G)$ colors.

Thus, we have $rvd(G)\leq \Delta(G)$ for any $K_4$-minor free graph $G$.
The bound is sharp for $G=K_{2,n-2}$, which is a $K_4$-minor free graph with $\Delta(G)=n-2$ and $rvd(G)=n-2$.

\end{proof}

\section{Hardness results}
Decide Rainbow Vertex-disconnection Coloring Problem (RVD-Problem)

Instance: A graph $G=(V,E)$ and a positive integer $k$.

Question: Does $G$ have a rainbow vertex-disconnection coloring using $k$ colors?

A graph $G$ is $k$-colorable if there exists a vertex-coloring $c:V(G)\rightarrow [k]$ such that no two adjacent vertices have the same color.
The coloring $c$ is called \emph{proper}.
The \emph{chromatic number} $\chi(G)$ of $G$ is the minimum $k$ such that $G$ is $k$-colorable.

\begin{thm}\label{hardbi}
RVD-Problem is NP-complete for bipartite graphs.
\end{thm}
\begin{proof}
Given a fixed $k$-vertex-coloring $c_0$ of a bipartite graph, it is polynomial time to vertify whether it is a rainbow vertex-disconnection coloring.
So RVD-Problem is in NP for bipartite graphs.

We give a polynomial reduction from the proper coloring problem of $G=(V,E)$,
which is NP-complete for general graphs. We will construct a graph $\widetilde{G}$ from $G$ such that $\chi(G)\leq k$ if and only if $rvd(\widetilde{G})\leq k+2|E|$.

We construct $\widetilde{G}$ as follows.
For each edge $uv$ in $G$, add four vertices $s_{uv},s'_{uv},t_{uv},t'_{uv}$ and replace $uv$ with edges $us_{uv},us'_{uv},vs_{uv},vs'_{uv}$.
Let $V_s=\{s_{uv},s'_{uv}:uv\in E\}$
and $V_t=\{t_{uv},t'_{uv}:uv\in E\}$.
Add $E(V_s,V_t)$.
Then we obtain the graph $\widetilde{G}$ with $V(\widetilde{G})=V\cup V_s\cup V_t$
and $E(\widetilde{G})=\{us_{uv},us'_{uv},vs_{uv},vs'_{uv}:uv\in E\}\cup E(V_s,V_t)$. Obviously, $\widetilde{G}$ is bipartite and we can construct it from $G$ in polynomial time.

If $\chi(G)\leq k$, then we give a proper coloring $c$ of $G$:
$V\rightarrow[k]$.
Let $\widetilde{c}$: $V(\widetilde{G})\rightarrow [k+2|E|]$ be a vertex-coloring of $\widetilde{G}$ as follows. For $v\in V$, $\widetilde{c}(v)=c(v)$. Let $V_s$ be rainbow using colors $\{k+1,k+2,\cdots,k+2|E|\}$
and $\widetilde{c}(t_{uv})=\widetilde{c}(s_{uv})$ for each $uv\in E$.
Then for $v\in V\cup V_t$, $N_{\widetilde{G}}(v)$ is rainbow.
For $s_{uv},s'_{uv}\in V_s$, $s_{uv}$ and $s'_{uv}$ have two neighbors $u$ and $v$ from $V$ in $\widetilde{G}$.
Since $c(u)\neq c(v)$, we have $\widetilde{c}(u)\neq \widetilde{c}(v)$.
So $N_{\widetilde{G}}(s_{uv})$, $N_{\widetilde{G}}(s'_{uv})$ are rainbow.
Thus, $\widetilde{c}$ is a rainbow vertex-disconnection coloring of $\widetilde{G}$ and $rvd(\widetilde{G})\leq k+2|E|$.

Conversely, assume that $rvd(\widetilde{G})\leq k+2|E|$.
Let $\widetilde{c}$: $V(\widetilde{G})\rightarrow [k+2|E|]$ be a rainbow vertex-disconnection coloring of $\widetilde{G}$.
Since any two vertices in $V_t$ has at least two common neighbors in $V_s$,
by Lemma \ref{rvddifcolor}, $V_t$ is rainbow.
For any vertex $x\in V$ and any vertex $y\in V_t$, assuming that $xz\in E$,
vertices $s_{xz}$ and $s'_{xz}$ in $V_s$ are two common neighbors of vertices $x$ and $y$. So the colors of $V$ in $\widetilde{G}$ are disjoint with the colors of $V_t$ in $\widetilde{G}$.
Let $\widetilde{c}_V$ be the coloring of $G$ by restricting $\widetilde{c}$ to $V$. Then $\widetilde{c}_V$ has at most $k$ colors.
For any two adjacent vertices $u$ and $v$ in $G$,
since $u$ and $v$ have two common neighbors $s_{uv},s'_{uv}$ in $\widetilde{G}$,
we have $\widetilde{c}_V(u)\neq \widetilde{c}_V(v)$ by Lemma \ref{rvddifcolor}.
So $\widetilde{c}_V$ is a proper coloring of $G$ and $\chi(G)\leq k$.
\end{proof}

A subset $S$ of $V(G)$ is called an \emph{independent set} of $G$ if no two vertices of $S$ are adjacent in $G$. An independent set is \emph{maximum} if $G$ has no independent set $S'$ with $|S'|>|S|$. The number of vertices in a maximum independent set of $G$ is called the \emph{independence number} of $G$ and is denoted by $\alpha(G)$.

A \emph{$k$-fold coloring} of a graph $G$ is an assignment of sets of size $k$ to vertices of a graph such that adjacent vertices receive disjoint sets.

The \emph{$k$-fold chromatic number}, denoted by $\chi_k(G)$, is the minimum number of colors to obtain a $k$-fold coloring of $G$.
The \emph{fractional chromatic number} of $G$ is defined as $\chi_{f}(G)=\inf_k\frac{\chi_k(G)}{k}$.
It has been proved that $\chi_{f}(G)\geq \frac{|V(G)|}{\alpha(G)}$ \cite{LPU}.
\begin{thm}
If $ZPP\neq NP$, then,
for every $\epsilon>0$,
it is not possible to efficiently approximate $rvd(G)$ within a factor of $n^{\frac{1}{3}-\epsilon}$,
for any bipartite graph $G$.
\end{thm}

\begin{proof}
For a given graph $G$ and any fixed $\epsilon>0$,
the problem of deciding whether $\chi(G)\leq n^\epsilon$ or $\alpha(G)<n^\epsilon$ is not possible in polynomial time,
unless $ZPP=NP$, where $n$ is the order of $G$ \cite{FK}.

We replace $V$ in $\widetilde{G}$ from Theorem \ref{hardbi} with $k$ copies of $V$, denoted by $V_j$ ($j\in [k]$).
Assume that $V=\{v_1,v_2,\cdots,v_n\}$ and $V_j=\{v^j_1,v^j_2,\cdots,v^j_n\}$ ($j\in [k]$).
We construct a new graph $H$ from $G=(V,E)$ with $V(H)=\bigcup_{j=1}^{k}V_j\cup V_s\cup V_t$ and $E(H)=\{u^js_{uv},u^js'_{uv},v^js_{uv},v^js'_{uv}:uv\in E, j\in [k]\}\cup E(V_s,V_t)$,
where $V_s=\{s_{uv},s'_{uv}:uv\in E\}$
and $V_t=\{t_{uv},t'_{uv}:uv\in E\}$.
Similarly to the proof of Theorem \ref{hardbi},
we have $rvd(H)\leq k\cdot\chi(G)+2|E|$.

Let $c_H$ be an rvd-coloring of $H$.
Let $Q_{uv}=\bigcup_{j=1}^{k}u^j\cup\bigcup_{j=1}^{k}v^j$ for $uv\in E$.
Since any two vertices in $Q_{uv}$ have two common neighbors $s_{uv}$ and $s'_{uv}$,
$Q_{uv}$ is rainbow under $c_H$ by Lemma \ref{rvddifcolor}.
Let $c_G$ be a vertex-coloring of $G$ such that
$c_G(v_i)=\{c_H(v^1_i),c_H(v^2_i),\cdots,c_H(v^k_i)\}$ ($i\in [n]$).
Then for $uv\in E$, $c_G(u)$ is disjoint with $c_G(v)$.
So $c_G$ is a $k$-fold coloring of $G$.
Since any vertex in $\bigcup_{j=1}^{k}V_j$ and any vertex in $V_t$ have two common neighbors in $V_s$,
by Lemma \ref{rvddifcolor},
the colors of $\bigcup_{j=1}^{k}V_j$ are disjoint with the colors of $V_t$ under $c_H$.
So $c_G$ has at most $rvd(H)-2|E|$ colors.
We have $\chi_k(G)\leq rvd(H)-2|E|$.
Thus, we obtain
$$\frac{kn}{\alpha(G)}+2|E|\leq k\cdot\chi_f(G)+2|E|\leq\chi_k(G)+2|E|\leq rvd(H)\leq k\cdot\chi(G)+2|E|.$$

If $\chi(G)\leq n^\epsilon$,
we have $rvd(H)\leq kn^\epsilon+2|E|$.
If $\alpha(G)< n^\epsilon$,
we have $rvd(H)> kn^{1-\epsilon}+2|E|$.
Choose $k=|E|$.
For $n\geq 4^{\frac{1}{\epsilon}}$,
we obtain
$$\frac{kn^{1-\epsilon}+2|E|}{kn^\epsilon+2|E|}
=\frac{n^{1-\epsilon}+2}{n^\epsilon+2}
\geq \frac{1}{4}n^{1-2\epsilon}
\geq n^{1-3\epsilon}\geq (|E|n+4|E|)^{\frac{1}{3}(1-3\epsilon)}=N^{\frac{1}{3}-\epsilon},$$
where $N=|V(H)|$.

So if we can efficiently $N^{\frac{1}{3}-\epsilon}$-approximate the rvd-coloring of $H$ then it is possible to efficiently decide whether $\chi(G)\leq n^\epsilon$ or $\alpha(G)<n^\epsilon$.
\end{proof}

\begin{thm}\label{hardsplit}
RVD-Problem is NP-complete for split graphs.
\end{thm}
\begin{proof}
Given a fixed $k$-vertex-coloring $c_0$ of a split graph, it is polynomial time to vertify whether it is a rainbow vertex-disconnection coloring.
So RVD-Problem is in NP for split graphs.

We give a polynomial reduction from the proper coloring problem of $G=(V,E)$,
which is NP-complete for general graphs. We will construct a graph $\widetilde{G}$ from $G$ such that $\chi(G)\leq k$ if and only if $rvd(\widetilde{G})\leq k+3|E|$.

We construct $\widetilde{G}$ as follows.
For each edge $uv$ in $G$, add three vertices $s_{uv},s'_{uv},s''_{uv}$ and replace $uv$ with edges $us_{uv},us'_{uv},us''_{uv},vs_{uv},vs'_{uv},vs''_{uv}$.
Let $V_s=\{s_{uv},s'_{uv},s''_{uv}:uv\in E\}$.
Add edges such that $V_s$ forms a clique.
Then we obtain the graph $\widetilde{G}$ with $V(\widetilde{G})=V\cup V_s$
and $E(\widetilde{G})=\{us_{uv},us'_{uv},us''_{uv},vs_{uv},vs'_{uv},vs''_{uv}:uv\in E\}\cup \{uv:u,v\in V_s\}$. Obviously, $\widetilde{G}$ is a split graph with clique $G[V_s]$ and independent set $V$. We can construct it from $G$ in polynomial time.

If $\chi(G)\leq k$, then we give a proper coloring $c$ of $G$:
$V\rightarrow[k]$.
Let $\widetilde{c}$: $V(\widetilde{G})\rightarrow [k+3|E|]$ be a vertex-coloring of $\widetilde{G}$ as follows. For $v\in V$, $\widetilde{c}(v)=c(v)$.
Let $V_s$ be rainbow using colors $\{k+1,k+2,\cdots,k+3|E|\}$.
Then for $v\in V$, $N_{\widetilde{G}}(v)$ is rainbow.
For $s_{uv}\in V_s$, $s_{uv}$ has two neighbors $u$ and $v$ in $V$.
Since $c(u)\neq c(v)$, we have $\widetilde{c}(u)\neq \widetilde{c}(v)$.
So $N_{\widetilde{G}}(s_{uv})$ is rainbow.
Similarly, $N_{\widetilde{G}}(s'_{uv})$, $N_{\widetilde{G}}(s''_{uv})$ are rainbow.
Thus, $\widetilde{c}$ is a rainbow vertex-disconnection coloring of $\widetilde{G}$ and $rvd(\widetilde{G})\leq k+3|E|$.

Conversely, assume that $rvd(\widetilde{G})\leq k+3|E|$.
Let $\widetilde{c}$: $V(\widetilde{G})\rightarrow [k+3|E|]$ be a rainbow vertex-disconnection coloring of $\widetilde{G}$.
If $|E|=1$, assuming the edge is $pq$ in $G$, then $V_s$ is a $p$-$q$ vertex-cut.
So $V_s$ is rainbow.
Since any two vertices in $V_s$ have at least two common neighbors in $V_s$ for $|E|\geq 2$,
by Lemma \ref{rvddifcolor}, $V_s$ is rainbow.
For any vertex $x\in V$ and any vertex $y\in V_s$, assuming that $xz\in E$,
there are at least two common neighbors of $x$ and $y$ from $\{s_{xz},s'_{xz},s''_{xz}\}$. So the colors of $V$ in $\widetilde{G}$ are disjoint with the colors of $V_s$ in $\widetilde{G}$.
Let $\widetilde{c}_V$ be the coloring of $G$ by restricting $\widetilde{c}$ to $V$. Then $\widetilde{c}_V$ has at most $k$ colors.
For any two adjacent vertices $u$ and $v$ in $G$,
since $u$ and $v$ have three common neighbors $s_{uv},s'_{uv},s''_{uv}$ in $\widetilde{G}$,
we have $\widetilde{c}_V(u)\neq \widetilde{c}_V(v)$ by Lemma \ref{rvddifcolor}.
So $\widetilde{c}_V$ is a proper coloring of $G$ and $\chi(G)\leq k$.
\end{proof}

\begin{thm}
If $ZPP\neq NP$, then,
for every $\epsilon>0$,
it is not possible to efficiently approximate $rvd(G)$ within a factor of $n^{\frac{1}{3}-\epsilon}$,
for any split graph $G$.
\end{thm}
\begin{proof}
We replace $V$ in $\widetilde{G}$ from Theorem \ref{hardsplit} with $k$ copies of $V$, denoted by $V_j$ ($j\in [k]$).
Assume that $V=\{v_1,v_2,\cdots,v_n\}$ and $V_j=\{v^j_1,v^j_2,\cdots,v^j_n\}$ ($j\in [k]$).
We construct a new graph $H$ from $G=(V,E)$ with $V(H)=\bigcup_{j=1}^{k}V_j\cup V_s$ and $E(H)=\{u^js_{uv},u^js'_{uv},u^js''_{uv},v^js_{uv},v^js'_{uv},v^js''_{uv}:uv\in E, j\in [k]\}\cup \{uv:u,v\in V_s\}$,
where $V_s=\{s_{uv},s'_{uv},s''_{uv}:uv\in E\}$.
Similarly to the proof of Theorem \ref{hardsplit},
we have $rvd(H)\leq k\cdot\chi(G)+3|E|$.

Let $c_H$ be an rvd-coloring of $H$.
Let $Q_{uv}=\bigcup_{j=1}^{k}u^j\cup\bigcup_{j=1}^{k}v^j$ for $uv\in E$.
Since any two vertices in $Q_{uv}$ have three common neighbors $s_{uv}$, $s'_{uv}$ and $s''_{uv}$,
$Q_{uv}$ is rainbow under $c_H$ by Lemma \ref{rvddifcolor}.
Let $c_G$ be a vertex-coloring of $G$ such that
$c_G(v_i)=\{c_H(v^1_i),c_H(v^2_i),\cdots,c_H(v^k_i)\}$ ($i\in [n]$).
Then for $uv\in E$, $c_G(u)$ is disjoint with $c_G(v)$.
So $c_G$ is a $k$-fold coloring of $G$.
Since any vertex in $\bigcup_{j=1}^{k}V_j$ and any vertex in $V_s$ have at least two common neighbors in $V_s$,
by Lemma \ref{rvddifcolor},
the colors of $\bigcup_{j=1}^{k}V_j$ are disjoint with the colors of $V_s$ under $c_H$.
So $c_G$ has at most $rvd(H)-3|E|$ colors.
We have $\chi_k(G)\leq rvd(H)-3|E|$.
Thus, we obtain
$$\frac{kn}{\alpha(G)}+3|E|\leq k\cdot\chi_f(G)+3|E|\leq\chi_k(G)+3|E|\leq rvd(H)\leq k\cdot\chi(G)+3|E|.$$

If $\chi(G)\leq n^\epsilon$,
we have $rvd(H)\leq kn^\epsilon+3|E|$.
If $\alpha(G)< n^\epsilon$,
we have $rvd(H)> kn^{1-\epsilon}+3|E|$.
Choose $k=|E|$.
For $n\geq 6^{\frac{1}{\epsilon}}$,
we obtain
$$\frac{kn^{1-\epsilon}+3|E|}{kn^\epsilon+3|E|}
=\frac{n^{1-\epsilon}+3}{n^\epsilon+3}
\geq \frac{1}{6}n^{1-2\epsilon}
\geq n^{1-3\epsilon}\geq (|E|n+3|E|)^{\frac{1}{3}(1-3\epsilon)}=N^{\frac{1}{3}-\epsilon},$$
where $N=|V(H)|$.

So if we can efficiently $N^{\frac{1}{3}-\epsilon}$-approximate the rvd-coloring of $H$ then it is possible to efficiently decide whether $\chi(G)\leq n^\epsilon$ or $\alpha(G)<n^\epsilon$.
\end{proof}

\section*{Conflict of interest}
The authors declare that they have no conflict of interest.

\section*{Data availability statement}
My manuscript has no associated data.

\end{document}